\documentclass[12pt]{amsart}
\usepackage{amsmath,amssymb,latexsym,cite}
\usepackage[small]{caption}
\usepackage{graphicx,wasysym,overpic,tikz,color}
\usepackage{amsfonts, amsmath, amssymb, amsthm, color,graphicx}
\usepackage{subfigure}

\usepackage[colorlinks=true,urlcolor=blue,
citecolor=red,linkcolor=blue,linktocpage,pdfpagelabels,bookmarksnumbered,bookmarksopen]{hyperref}
\usepackage[english]{babel}

\usepackage[left=2.6cm,right=2.6cm,top=2.7cm,bottom=2.7cm]{geometry}

\numberwithin{equation}{section}
\newtheorem{theorem}{Theorem}[section]
\newtheorem{proposition}[theorem]{Proposition}
\newtheorem{lemma}[theorem]{Lemma}

\theoremstyle{definition}

  \newcommand{\rr}{\mathbb R}
\newcommand{\eps}{\epsilon}

\renewcommand{\(}{\left(}
\renewcommand{\)}{\right)}

\def\sideremark#1{\ifvmode\leavevmode\fi\vadjust{\vbox to0pt{\vss
 \hbox to 0pt{\hskip\hsize\hskip1em
 \vbox{\hsize2.1cm\tiny\raggedright\pretolerance10000
  \noindent #1\hfill}\hss}\vbox to15pt{\vfil}\vss}}}%

\begin{document}

\title[Blow-up for sign-changing solutions of the critical heat equation]{Blow-up for sign-changing solutions of the critical heat equation  in domains with a small hole}

\author[I.\ Ianni]{Isabella Ianni}
\author[M.\ Musso]{Monica Musso}
\author[A.\ Pistoia]{Angela Pistoia}

\address[Isabella Ianni]{
Dipartimento di Matematica e Fisica,
Seconda Universit\`a degli Studi di Napoli,
Viale Lincoln 5, 81100 Caserta, Italy}
\email{isabella.ianni@unina2.it}

\address[Monica Musso] {Departamento de Matematicas,
Pontificia Universidad Cat\'olica de Chile
Casilla 306, Correo 22 Santiago, Chile. }
\email{mmusso@mat.puc.cl}

\address[Angela Pistoia] {Dipartimento SBAI, Universit\`{a} di Roma ``La Sapienza", via Antonio Scarpa 16, 00161 Roma, Italy}
\email{pistoia@dmmm.uniroma1.it}

\begin{abstract}
We consider  the critical heat equation
\begin{equation} \label{CH}\tag{CH}
\begin{array}{lr}
v_t-\Delta v =|v|^{\frac{4}{n-2}}v & \Omega_{\epsilon}\times (0, +\infty)
\\
v=0 & \partial\Omega_{\epsilon}\times (0, +\infty)
\\
v=v_0 &  \mbox{ in } \Omega_{\epsilon}\times \{t=0\}
\end{array}
\end{equation}
in  $\Omega_{\epsilon}:=\Omega\setminus B_{\epsilon}(x_0)$ where $\Omega$ is a smooth bounded domain in $\mathbb R^N$, $N\geq 3$ and $B_{\epsilon}(x_0)$ is a ball of $\mathbb R^N$ of center $x_0\in\Omega$ and radius $\epsilon >0$ small.
\\
We show that if  $\epsilon>0$ is small enough, then there exists a sign-changing stationary solution $\phi_{\epsilon}$ of \eqref{CH} such that the solution of \eqref{CH} with initial value $v_0=\lambda \phi_{\epsilon}$ blows up in finite time if $|\lambda -1|>0$ is sufficiently small.\\
This shows in particular that the set of the initial conditions for which the solution of \eqref{CH} is global and bounded is not star-shaped.
\end{abstract}


\date{\today}

\maketitle
  \footnotetext{The second author has been supported by Fondecyt Grant 1120151. The first and the third authors have been partially supported by the Gruppo Nazionale per l'Analisi Matematica, la Probabilit\'a e le loro Applicazioni (GNAMPA) of the Istituto Nazionale di Alta Matematica (INdAM) and  PRIN 2012 grant(Italy).}

\section{Introduction}

We consider the semilinear parabolic equation
\begin{equation}\label{problemGenDomain}\left\{\begin{array}{ll}v_t-\Delta v= |v|^{p-1}v & \mbox{ in }D\times (0,T)\\
v=0 & \mbox{ on }\partial D\times (0,T)\\
v(0)=v_0 & \mbox{ in }D
\end{array}\right.
\end{equation}
where $D$ is a smooth bounded domain in $\mathbb R^N$, $N\in\mathbb N$, $N\geq 3$ and  $p>1$.
\\
\\
For any $p>1$ problem \eqref{problemGenDomain} is locally well-posed for $v_0\in C_0(D)$, where
\[C_0(D)=\{v\in C(\bar D),\ v=0  \mbox{ on } \partial D\}.\]
Let $T_{max}(v_0)\in (0, +\infty]$ denote the maximal existence time of the unique local in time classical solution $v=v(\cdot, t)$ of \eqref{problemGenDomain}.\\
\\
The solution $v$ is said to be {\it global} when $T_{max}(v_0)=\infty$, while when $T_{max}(v_0)<\infty$ it is said to {\it blow-up in finite time}.
\\
Let us define the set of initial values such that the solution is global
\[\mathcal{G}=\{v_0\in C_0(D),\ T_{max}(v_0)=+\infty\}\]
and its complementary set of initial conditions for which the corresponding solution blows-up in finite time
\[\mathcal{F}=\{v_0\in C_0(D),\ T_{max}(v_0)<+\infty\}.\]
Let also $\mathcal{B}\subseteq\mathcal{G}$ be the set of initial values for which the solution  is global in time and has an $L^{\infty}$-global bound
\[\mathcal{B}=\{v_0\in C_0(D),\ T_{max}(v_0)=+\infty\mbox{ and } \sup_{t\geq 0}\|v(t)\|_{L^{\infty}}<\infty\}.\]
When $p$ is subcritical,  namely $1<p<p_S$ where
\begin{equation}
p_S=\frac{N+2}{N-2}
\end{equation}
($2^*=p_S+1$ is the critical Sobolev exponent), one has that $\mathcal{B}=\mathcal{G}$ since
\[
T_{max}(v_0)=+\infty\ \Rightarrow \sup_{t\geq 0}\|v(t)\|_{L^\infty}< \infty,
\]
but for $p\geq p_S$ it may occur that $\mathcal{B}\subset\mathcal{G}$.  In the critical case
$p=p_S$ for instance, it is well known that {\it infinite time blow-up} may occur, namely there may exist $v_0\in C_0(D)$ such that
\[
T_{max}(v_0)=+\infty\ \mbox{ and }\ \lim_{t\uparrow +\infty}\|v(t)\|_{L^\infty}=+\infty.
\]
(cfr. \cite{NST, GV, GK} for a radial positive $v_0\in\mathcal{G}\setminus\mathcal{B}$ when $D$ is a ball, and \cite{Suzuki} for a positive $v_0\in\mathcal{G}\setminus\mathcal{B}$ when $D$ is convex and symmetric, see also \cite{Ishiwata} for necessary and sufficient conditions for the  $L^{\infty}$ global bound in the critical and subcritical case).

\

Let us observe that all the stationary solutions of \eqref{problemGenDomain} belong to $\mathcal{B}$ (if any: when $p\geq p_S$ and $D$ is star-shaped for instance the only stationary solution is the trivial one, cfr. \cite{Pohozaev}), moreover $\mathcal{B}$  contains a neighborhood of the origin (since the zero solution is exponentially asymptotically stable in $L^{\infty}$, see \cite[Theorem 19.2]{quittnersouplet}).
\\

If we restrict ourselves to non-negative initial data, then the solutions are positive by the parabolic maximum principle, hence we may replace the nonlinearity $|v|^{p-1}v$ by $|v|^p$ which is convex and so (see \cite{Lions}) the corresponding sets
$\mathcal{G}^+=\{v_0\in\mathcal G,\ v_0\geq 0\}$ and  $\mathcal{B}^+= \{v_0\in\mathcal B,\ v_0\geq 0\}$ are convex, hence star-shaped  around $0$.

More specifically if $\phi$ is a stationary positive solution to \eqref{problemGenDomain}  and $v_0=\lambda\phi$ for $\lambda >0$, then
$v_0>\phi$ if $\lambda >1$  and so $v_0\in \mathcal{F}$ (see for instance \cite[Theorem 17.8]{quittnersouplet}), while $0\leq v_0\leq\phi$ for $0<\lambda \leq 1$ and so $v_0\in \mathcal B$ (since by the parabolic maximum principle $0<v(t)\leq\phi$ for $t>0$ and moreover \cite[Lemma 17.9]{quittnersouplet} applies).

\

If we consider sign-changing initial data then the arguments above can not be applied. In particular if $\phi$ is a sign-changing stationary solution to \eqref{problemGenDomain}, then it is not comparable with $\lambda\phi$ for $\lambda\neq 1$. Anyway, since the zero solution is exponentially asymptotically stable in $L^{\infty}$, then clearly $\lambda\phi\in \mathcal{B}$  for $\lambda$ sufficiently small, moreover it is also known that $\lambda\phi\in \mathcal{F}$ for $\lambda$ sufficiently large (see \cite[Theorem 17.6]{quittnersouplet}) and of course if $\lambda=1$ then $\lambda\phi\in \mathcal B$.

\

Recently it has been proved (cfr. \cite{CazenaveDicksteinWeissler}  when $D$ is a ball, \cite{mps} for any smooth bounded domain $D$), that if $p$ is subcritical and sufficiently close to the critical exponent, then there exist  sign-changing stationary solutions $\phi$ such that $\lambda\phi\in\mathcal{F}$ for $\lambda >0$ sufficiently close to $1$. This results shows that in the subcritical case $\mathcal{B}$ in general is not star-shaped around 0, hence not convex.

\

The aim of this paper is to extend this result to the critical case $p=p_S$.

\

For $p=p_S$  already the existence of  a (sign-changing) stationary solution $\phi$ of  \eqref{problemGenDomain} is an issue.
Indeed it is well known that there are no nontrivial stationary solutions when $D$ is strictly starshaped (cfr. \cite{Pohozaev}), while it is easy to prove the existence of infinitely many radial stationary solutions if $D$ is an annulus (cfr. \cite{KazdanWarner}). It is also known that there is a positive stationary solution whenever the homology of dimension $d$ of $D$ with $\mathbb Z_2$ coefficients is nontrivial for some positive integer $d$ (cfr. \cite{BahriCoron}).

\

We consider here problem \eqref{problemGenDomain} when $p=p_S$ and  $D$ is a  domain with a small hole, precisely
\begin{equation}
D:=\Omega_{\epsilon}=\Omega\setminus B_{\epsilon}(x_0),
\end{equation}
where $\Omega$ is a smooth bounded domain in $\mathbb R^N$, $N\in\mathbb N$, $N\geq 3$ and $B_{\epsilon}(x_0)\subset\Omega$ is a ball of $\mathbb R^N$ of center $x_0\in\Omega$ and radius $\epsilon >0$ small enough.
\\

Under these assumptions the existence of a positive stationary solution is a classical result by Coron (cfr. \cite{Coron}), while the existence of an arbitrary large number of sign-changing stationary solutions has been obtained in  more recent works by Musso and Pistoia (cfr. \cite{MussoPistoia})
 and Ge, Musso and Pistoia (cfr. \cite{gmp}), in the case when the hole is small enough.

\

Our main result is the following
\begin{theorem}
\label{teorema:casoGenerale}
Let $\Omega_{\epsilon}:=\Omega\setminus B_{\epsilon}(x_0)$ where $\Omega$ is a smooth bounded domain in $\mathbb R^N$, $N\in\mathbb N$, $N\geq 3$, $x_0\in\Omega$ and $\epsilon >0$ small.\\
There exists $\epsilon_0 >0$ with the following property. If $0<\epsilon\leq \epsilon_0$ then there exists a sign-changing stationary solution $\phi_{\epsilon}$ of
\begin{equation}
\label{problem}
\left\{\begin{array}{ll}v_t-\Delta v= |v|^{\frac{4}{N-2}}v & \mbox{ in }\Omega_{\epsilon}\times (0,T)\\
v=0 & \mbox{ on }\partial \Omega_{\epsilon}\times (0,T)\\
v(0)=v_0 & \mbox{ in }\Omega_{\epsilon}
\end{array}\right.
\end{equation}
 and a constant $\delta_{\epsilon}>0$ such that if $\lambda>0$, $0<|\lambda -1|<\delta_{\epsilon}$ then the classical solution $v$ of \eqref{problem} with initial value $v_0=\lambda \phi_{\epsilon}$ blows up, namely $v_0\in \mathcal F \cup (\mathcal G\setminus\mathcal B)$.\\
 As a consequence  $\mathcal B$ is not star-shaped around zero.
\end{theorem}

Observe that now,  unlike the subcritical case, the solution $v$ may blow-up in finite or infinite time and so we can not conclude that also the set $\mathcal{G}$ is not convex, unless we restrict to more specific situations like in the following result related to sign-changing radial solutions in an annulus

\begin{theorem}
\label{teorema:casoRadiale}
Let $\Omega_{\epsilon}:=\{x\in \mathbb R^N:\ 0<\epsilon <|x|<1\}$, $N\in\mathbb N$, $N\geq 3$,  $\epsilon>0$ small.

There exists $\epsilon_0 >0$ with the following property. If $0<\epsilon\leq \epsilon_0$ then there exists a sign-changing radial stationary solution $\widehat{\phi_{\epsilon}}$ of  \eqref{problem} and a constant $\delta_{\epsilon}>0$ such that if $\lambda>0$, $0<|\lambda -1|<\delta_{\epsilon}$ then the classical solution $v$ of \eqref{problem} with initial value $v_0=\lambda \widehat{\phi_{\epsilon}}$ blows-up in finite time, namely $v_0\in \mathcal F$.
\\
As a consequence both $\mathcal B$ and $\mathcal G$ are not star-shaped around zero.
\end{theorem}

\

The sign-changing stationary solution $\phi_{\epsilon}$ of Theorem \ref{teorema:casoGenerale} is any bubble tower solution found  in \cite{gmp}. The proof consists then in scaling it properly and performing an asymptotic spectral analysis of the linearized problem,  similarly as it has been done in the almost critical case (\cite{CazenaveDicksteinWeissler, mps}, see also \cite{DicksteinPacellaSciunzi, DeMarchisIanni} for the case $N=2$). Now the exponent of the nonlinearity is fixed, and the scaling parameter depends only on the radius $\epsilon$ of the hole. The asymptotic analysis is possible again thanks to the knowledge of the limit problem. Combining it with  general results for the heat flow (see Proposition \ref{proposition:generaleCazenaveDickWeiss} in Section \ref{Section:preliminaries}), we can show that $\phi_{\epsilon}$ can be compared with the solution $v$ at a certain time $t_0>0$.
The blow-up result then follows from a blow-up criterion via comparison for sign changing solutions of the critical heat equation (Proposition \ref{proposition:comparison}).
\\

The proof of Theorem \ref{teorema:casoRadiale} is obtained by repeating similar arguments but starting from any of the sign-changing  \emph{radial} bubble tower stationary solutions $\widehat{\phi_{\epsilon}}$  found in \cite{MussoPistoia}. With this choice the solution $v$ is radial and so if it is global then it must satisfy an $L^{\infty}$ global bound (cfr. \cite{Ishiwata}), thus excluding   infinite time blow-up.

\

\section{Preliminaries}\label{Section:preliminaries}

In this section we provide a blow-up criterion via comparison  for sign-changing solutions of the critical heat equation. It extends to the critical case the analogous result already known  for the subcritical heat equation (see \cite[Proposition B.1]{CazenaveDicksteinWeissler} and \cite[Theorem 10]{GazzolaWeth}).
Unlike the subcritical case, both finite time blow-up and infinite time blow up can in general occur now.

\

\begin{proposition}\label{proposition:comparison}
Let $\psi\in C_{0}(D)$ be a sign-changing stationary solution of \eqref{problemGenDomain} with $p=p_S$. Let $v_0\in C_0(D)$, $v_0\not\equiv \psi$ be either $v_0\geq\psi$ or $v_0\leq\psi$, then $v_0\in\mathcal{F}\cup (\mathcal G\setminus\mathcal B).$\\
If in particular $D$ is an annulus $\{x\in\mathbb R^N:\ a<|x|<b\}$ (for $b>a>0$) and $v_0$ is radially symmetric, then $v_0\in\mathcal F$.
\end{proposition}
\begin{proof}
Once the first part is proved, the last assertion follows directly from \cite{Ishiwata}, where it has been showed that a necessary and sufficient condition to get an $L^{\infty}$ global bound for a global in time solution $v$ of \eqref{problemGenDomain} with $p\in(1,p_S]$ is that the energy functional
\[J_p(u)=\frac{1}{2}\|\nabla u\|_{L^2(D)}^2-\frac{1}{p+1}\|u\|_{L^{p+1}(D)}^{p+1}, \quad u\in H^1_0(D)\]
satifies the Palais-Smale condition along $v$. Indeed when the domain $D$ in an annulus and $v_0$ is radial then the solution $v$  of \eqref{problemGenDomain} with initial condition $v(0)=v_0$ is radial and the Sobolev compact embedding $H^1_{0,r}(D)\subset\subset   L^{2^*}(D)$ (where $H^1_{0,r}(D)$ is the subspace of  the radial functions in $H^1_{0}(D)$) ensures that $J_{p_S}$ satisfies the Palais-Smale condition along $v$. \\

Next we prove the first part. The proof follows closely the one  for the subcritical case, the main difference being now the lack of a priori bounds for the $L^{\infty}$ norm of global solutions.\\

We repeat it in details for the reader convenience. We prove the case $v_0\geq\psi$, the other case being similar.\\

Let $v$ be the solution of \eqref{problemGenDomain} with initial condition
$v(0)=v_0$. Arguing by contradiction, assume that $v$ doesn't blow-up, namely that
\[T_{max}(v_0)=+\infty \ \ \mbox{ and }\ \ \sup_{t\geq 0}\|v(t)\|_{L^{\infty}}<\infty.\]
By the parabolic  {strong comparison principle}
\[v(x,t)>\psi(x),\quad x\in D,\ t>0,\]
hence, at time $t=1$, there exists $\epsilon_0>0$ such that
\[
v(x,1)>\psi(x)+\epsilon\varphi_1(x),\quad x\in D,\  0<\epsilon<\epsilon_0,
\]
where $\varphi_1$ is the first eigenfunction, $\varphi_1>0$, normalized by $\|\varphi_1\|_{L^{\infty}}=1$, of the linear operator $-\Delta-p_S|\psi(x)|^{p_S-1}$ in $D$ with Dirichlet boundary condition.

Let $w$ be the solution of \eqref{problemGenDomain} with initial condition $w(x,0)= \psi(x)+\epsilon\varphi_1(x)$, $x\in D$.
Since $w(x,0)<v(x,1)$, by the {weak comparison principle} it follows that
\begin{equation}\label{alpha}
w(x,t)\leq v(x, t+1),\quad t\geq 0, \ x\in D.
\end{equation}
Moreover, since $\psi<w(0)$ in $D$, by the {weak comparison principle}
\begin{equation}\label{beta}\psi(x)\leq w(x,t),\quad t\geq 0,\ x\in D.
\end{equation}
By \eqref{alpha} and \eqref{beta} it then follows that $w$ is global and also that $\|w(t)\|_{L^{\infty}}\leq max\{\|\psi\|_{L^{\infty}}, \|v(t+1)\|_{L^{\infty}}\}$, $t\geq 0$, and so that
\begin{equation}\label{boundinfinito}
\sup_{t\geq 0}\|w(t)\|_{L^{\infty}}<\infty.
\end{equation}
Since the function $w(x,0)=\psi(x)+\epsilon\varphi_1(x)$, for $\epsilon>0$  sufficiently small, is a stationary  subsolution for \eqref{problemGenDomain}  (see \cite[Lemma B.4]{CazenaveDicksteinWeissler}), it follows (cfr. \cite[Proposition 52.19]{quittnersouplet}) that $w_t\geq 0$ for $x\in D$, $t\geq 0$, namely $t\mapsto w(t)$ is monotone increasing.
\\
By \eqref{boundinfinito} and the monotonicity it follows that there exists $\psi' \in C_0(D)$, stationary solution of \eqref{problemGenDomain}, such that
\[
w(t)\uparrow \psi'\  \mbox{ in }C_0(D),\ \mbox{ as }t\rightarrow +\infty.\]
But then by the monotonicity, recalling that $\varphi_1\geq 0$
\begin{equation}\label{pre}
\psi'(x)\geq w(x,t)\geq w(x,0)=\psi(x)+\epsilon\varphi_1(x)\geq\psi(x), \quad\ x\in D,\ t>0\end{equation}
 from which it follows that $\psi'\not\equiv 0.$\\
 \\
Moreover, by the parabolic strong comparison principle,
\[\psi'(x)> w(x,t),\quad x\in D,\ t>0,\] indeed by \eqref{pre} $\psi'(x)\geq w(x,0)$ and moreover  $\psi'(x)\not\equiv w(x,0)$ (otherwise $\psi(x)+\epsilon\varphi_1(x)$ would be a stationary solution to \eqref{problemGenDomain}, which is not the case).
\\
Hence, at time $t=1$, there exists $\epsilon_0'>0$ such that
\[
\psi'(x)-\epsilon'\varphi'_1(x)\geq w(x,1),\quad x\in D,\  0<\epsilon'<\epsilon_0',
\]
where  $\varphi'_1$ is the first eigenfunction, $\varphi'_1>0$, normalized by $\|\varphi'_1\|_{L^{\infty}}=1$, of the linear operator $-\Delta-p_S|\psi'(x)|^{p_S-1}$ in $D$ with Dirichlet boundary condition.
\\
Since for $\epsilon'>0$ small enough the function $\psi'(x)-\epsilon' \varphi'_1(x)$ is a stationary supersolution for \eqref{problemGenDomain} (see \cite[Lemma B.4]{CazenaveDicksteinWeissler}), by the comparison principle we get
\[\psi'(x)-\epsilon'\varphi'_1(x)\geq w(x,t)\quad x\in D,\ t\geq 1\]
and passing to the limit as $t\rightarrow +\infty$ we get a contradiction.
\end{proof}

\

\

In order to prove Theorem \ref{teorema:casoGenerale} and Theorem \ref{teorema:casoRadiale} we will need the following general result, whose proof can be found in \cite{CazenaveDicksteinWeissler}:

\begin{proposition} \label{proposition:generaleCazenaveDickWeiss}
Let $\phi\in C_0(D)$ be a sign changing stationary solution of \eqref{problemGenDomain}
and let $\varphi_{1}$ be the positive eigenfunction of the self-adjoint operator $L$ given by $L\varphi=-\Delta \varphi-p|\phi|^{p-1}\varphi$, for $\varphi\in H^2(D)\cap H^1_0(D)$. For $\lambda >0$, let  $v^{\lambda}$ be the solution of \eqref{problemGenDomain} with the initial condition $v^{\lambda}(0)=\lambda\phi$. Assume that
\begin{equation}\label{condizioneProdottoScalareDiversoDaZeroGen} \int_D\phi\varphi_{1}> 0.
\end{equation}
Then there exist $t_0>0$ and $\delta >0$ such that
\[
\begin{array}{lr}
v^{\lambda}(t_0)>\phi,\quad\mbox{ for }\  \lambda\in(1, 1 +\delta],  \\
v^{\lambda}(t_0)<\phi,\quad\mbox{ for }\ \lambda\in[1-\delta,1).
\end{array}
\]
\end{proposition}

\begin{proof}
The proof consists in linearizing the equation \eqref{problemGenDomain} in $\phi$, by setting $z^{\lambda}$ through
\[(\lambda-1)z^{\lambda}(t)=u^{\lambda}(t)-\phi,\]
and than, by means of condition \eqref{condizioneProdottoScalareDiversoDaZeroGen} and the properties of linear equations, in showing the existence of $t_0>0$ and $\delta >0$, such that
\[z^{\lambda}(x,t_0)>0, \ \mbox{ for } |\lambda-1|\leq\delta.\]
We refer the reader to \cite{CazenaveDicksteinWeissler} for all the details.
\end{proof}

\

\section{Proof of Theorem \ref{teorema:casoGenerale}}\label{Section:proofGen}

The strategy of the proof of  Theorem \ref{teorema:casoGenerale} is similar to the one in \cite{CazenaveDicksteinWeissler, mps}: we show the existence of a sign-changing  stationary solution to \eqref{problem} which satisfies  the assumption \eqref{condizioneProdottoScalareDiversoDaZeroGen}. Then Proposition \ref{proposition:generaleCazenaveDickWeiss} applies and the conclusion follows from a comparison argument.\\
\\
In our case the comparison result is given by Proposition \ref{proposition:comparison}.\\
\\
As the sign-changing stationary solution we take  the \emph{$k$-tower} solution $\phi_{\epsilon}$ built in \cite{gmp} in  domains with a sufficiently  small hole (Lemma \ref{existenceOfTowers} below). Hence the core of the proof will be to show that a $k$-tower stationary solution satisfies the assumption \eqref{condizioneProdottoScalareDiversoDaZeroGen}, this is obtained by an asymptotic spectral analysis of the linearized operator $-\Delta-p_S|\phi_{\epsilon}|^{p_S-1}$ as $\epsilon$ goes to zero, and it is  the result in Proposition \ref{towerVerificaCondizioneIntegraleCasoGEN}  below.\\
\\
Before stating our results, we need to fix some notation.
 \\
Let
\begin{equation}\label{bubble}
U_{\delta , \xi} (x):= \alpha_N \left( \frac{\delta}{ \delta^2 + |x-\xi |^2}
\right)^{\frac{N-2}{2}},\ \ x\in\mathbb R^N
\end{equation}
where $\alpha_N:=[N(N-2)]^{\frac{N-2}{4}},$ $\delta$ is any positive
parameter and $\xi$ a point in $\rr^N$. These functions are
the only positive bounded solutions of the critical problem on the whole space
\begin{equation}\label{rn}\Delta u+ u^{p_S }   =0 \ {\mbox { in }} \ \rr^N.
\end{equation}

\

\begin{lemma}[Existence of $k$-tower stationary solutions]
\label{teo1}\label{existenceOfTowers}
Let $\Omega_{\epsilon}:=\Omega\setminus B_{\epsilon}(x_0)$ where $\Omega$ is a smooth bounded domain in $\mathbb R^N$, $N\in\mathbb N$, $N\geq 3$, $x_0\in\Omega$ and $\epsilon >0$ small.
For any integer $k\ge{2}$ there exists $\eps_k>0$ such that for any $\eps\in(0,\eps_k)$ problem \eqref{problem} has  a $k$-tower \textcolor{blue}{sign-changing} stationary solution $\phi_\eps$ whose profile is
\begin{equation}\label{pro}
\phi_\eps(x)=\sum\limits_{i=1}^k(-1)^{i}  U_{{\delta_i}_\eps,{\xi_i}_\eps}(x)+R_\eps(x),\ x\in\Omega_\eps
\end{equation}
where the concentration parameters ${\delta_i}_\eps$'s satisfy
\begin{equation}\label{de}
{\delta_i}_\eps:={d_i}_\eps\eps^{2i-1\over 2k},\ {d_i}_\eps\in\mathbb R\quad \hbox{and}\quad {d_i}_\eps\to d_i>0\ \hbox{as}\ \eps\to0\quad
\hbox{for}\  i=1,\dots,k,
\end{equation}
  the concentration points ${\xi_i}_\eps$'s satisfy
\begin{equation}\label{xi}
{\xi_i}_\eps:=x_0+{\delta_i}_\eps{\tau_i}_\eps,\quad {\tau_i}_\eps\in\rr^N \quad \hbox{and}\quad {\tau_i}_\eps\to\tau_i\ \hbox{as}\ \eps\to0\quad
\hbox{for}\  i=1,\dots,k,
\end{equation}
and the remainder term $R_\eps$ satisfies
\begin{equation}\label{resto}
\|R_\eps\|_{L^{2N\over N-2}(\Omega_\eps)}\to 0\ \ \hbox{as}\ \eps\to0.
\end{equation}
\end{lemma}

\begin{proof}
In  \cite{gmp} it was proved that    for any integer $k\ge1$
there exists $\eps_k>0$ such that for any $\eps\in(0,\eps_k)$ problem \eqref{problem} has  a stationary solution $\phi_\eps$ whose profile is
\begin{equation}\label{pq1}
\phi_\eps(x)=
\sum\limits_{i=1}^k(-1)^i  P_{\Omega_\eps}U_{{\delta_i}_\eps,{\xi_i}_\eps}(x)  +\psi_\eps(x),\ x\in\Omega_\eps
\end{equation}
where  the concentration parameters $\delta_1={\delta_1}_\eps,\dots,\delta_k={\delta_k}_\eps$  satisfy
\eqref{de} and the concentration points $\xi_1={\xi_1}_\eps,\dots,\xi_k={\xi_k}_\eps$  satisfy
\eqref{xi}.
Here $P_{\Omega_\eps}U_{\delta_i,\xi_i} $ denotes the projection of the bubble $ U_{\delta_i,\xi_i} $ onto $H^1_0(\Omega_\eps)$, namely the solution of
$$\Delta P_{\Omega_\eps}U_{\delta_i,\xi_i}=\Delta  U_{\delta_i,\xi_i}\ \hbox{in}\ \Omega_\eps,\quad   P_{\Omega_\eps}U_{\delta_i,\xi_i}=0\ \hbox{on}\ \partial\Omega_\eps.$$
Moreover,   the remainder term $\psi_\eps$ satisfies (see, for example, Proposition 2.1 in \cite{gmp})
\begin{equation}\label{pq2}
\|\psi_\eps\|_{H^1_0(\Omega_\eps)}\to0\ \hbox{as}\ \eps\to0.
\end{equation}
It is important to point out that the projection of the bubble has the following expansion (see, for example, Lemma 3.1 in \cite{gmp})
\begin{equation}\label{pq3}P_{\Omega_\eps}U_{\delta_i,\xi_i}(x)= U_{\delta_i,\xi_i}(x)-
\alpha_N{\delta_i } ^{N-2\over2} H(x,x_0)-\alpha_N {\eps^{N-2}\over\delta_i^{N-2\over2}(1+|\tau_i|^2)^{N-2\over2}}{1\over|x-x_0|^{N-2}}+{R_i}_\eps(x),\end{equation}
 {where the function  $H(x,y)$ in \eqref{pq3} is the regular part of  the Green function  $G(x,y)$ of the Laplace operator in
$\Omega$ with zero Dirichlet boundary condition,} and ${R_i}_\eps$ satisfies the pointwise estimate
\begin{equation}\label{pq4}|{R_i}_\eps(x)|\le c\ \delta_i^{N-2\over 2}
\left[ {\eps  ^{N-2}\over|x-x_0|^{N-2}}+\({\eps\over\delta_i}\)^{N-1}{1\over|x-x_0|^{N-2}}+\delta_i^2+\({\eps\over\delta_i}\)^{N-2}\right],\ x\in \Omega_\eps,
\end{equation}
for some positive constant $c.$

By \eqref{pq3} and \eqref{pq4}, taking into account that the function $H(\cdot,x_0)$ is bounded in $\Omega_\eps$ we deduce that
\begin{equation}\label{pq5}P_{\Omega_\eps}U_{\delta_i,\xi_i}(x)= U_{\delta_i,\xi_i}(x)+\bar {R_i}_\eps(x)\end{equation}
where $\bar {R_i}_\eps$ satisfies the pointwise estimate
\begin{equation}\label{pq6}|\bar {R_i}_\eps(x)|\le c\(\delta_i^{N-2\over 2}+{\eps^{N-2}\over\delta_i^{N-2\over2}}{1\over |x-x_0|^{N-2}} \), \ x\in \Omega_\eps, \end{equation}
A straightforward computation shows that
\begin{equation}\label{pq7}\left\|\bar {R_i}_\eps\right\|_{L^{2N\over N-2}(\Omega_\eps)}\to 0\ \hbox{as}\ \eps\to0, \end{equation}
because of \eqref{de}.
Finally, we set
$$R_\eps(x):=\sum\limits_{i=1}^k \bar {R_i}_\eps(x)+\psi_\eps(x),\quad x\in\Omega_\eps $$
and by \eqref{pq1}, \eqref{pq2}, \eqref{pq5} and \eqref{pq7} we deduce \eqref{pro} and \eqref{resto}.
\end{proof}

\

The main result of this section is the following

\begin{proposition}\label{towerVerificaCondizioneIntegraleCasoGEN}
Let $\Omega_{\epsilon}:=\Omega\setminus B_{\epsilon}(x_0)$, where $\Omega$ is a smooth bounded domain in $\mathbb R^N$, $N\in\mathbb N$, $N\geq 3$, $x_0\in\Omega$ and $\epsilon >0$. Let $\phi_{\epsilon}$ be as in Proposition \ref{existenceOfTowers}.
There exists $\epsilon_0>0$ such that for any $\epsilon\in (0,\epsilon_0)$
\[\int_{\Omega_{\epsilon}}\phi_{\epsilon}\varphi_{1,\epsilon}dx >0,\]
where $\varphi_{1,\epsilon}$ is the positive eigenfunction of the self-adjoint operator $L_{\epsilon}=-\Delta -p_S|\phi_{\epsilon}|^{p_S-1}$ on $L^2(\Omega_{\epsilon})$ with domain $ H^2(\Omega_{\epsilon})\cap H^1_0(\Omega_{\epsilon})$.
\end{proposition}

\

The proof of Proposition \ref{towerVerificaCondizioneIntegraleCasoGEN} relies on a spectral analysis of the linearized operator $-\Delta-p_S|\phi_{\epsilon}|^{p_S-1}$ for $\epsilon$ sufficiently small.  By a suitable scaling, we can pass to the limit as $\epsilon$ goes to zero and study the analogous spectral problem on $\mathbb R^N$.\\

\

Before proving Proposition \ref{towerVerificaCondizioneIntegraleCasoGEN} we need some preliminary results.

\

In order to simplify the notation we define
\begin{equation}
f(s):=|s|^{p_S-1}s, \ s\in\mathbb R.
\end{equation}
Moreover we set
\begin{equation}\label{U}
U:=U_{1,0}
\end{equation}
(see  \eqref{bubble} for the definition of $U_{\delta,\xi}$ for any $\delta>0$, $\xi\in\mathbb{R}^N$).

\

 Let us consider the linearization of the limit problem \eqref{rn} around $U$, namely the linear problem
$$\mathcal L ^*v:=-\Delta v-f'(U)v,\ v\in    H^1(\rr^N).$$
Let us define the first eigenvalue of $\mathcal L ^*$ by
\begin{equation}\label{las}
\lambda^*:=\inf\limits_{v\in H^1(\rr^N)\atop  \|v\|_{L^2(\rr^N)=1}} \int\limits_{\rr^N}\(|\nabla v|^2-f'(U)v^2\)dx  .\end{equation}
The following result holds true (see \cite{mps}).
\begin{lemma} \label{3.3}
There hold true
\begin{itemize}
\item[(i)] $\lambda^*\in(-\infty,0),$
\item[(ii)] there exists a unique positive minimizer $\varphi^*$ which is radial and radially nonincreasing. $\varphi^*$ is an eigenvector associated to $\lambda^*,$
\item[(iii)] every minimizing sequence has a subsequence which strongly converges in $L^2(\rr^N).$
\end{itemize}
\end{lemma}

Now, we fix an integer $k\ge1$ and we consider the $k-$tower solution $\phi_\eps$ found in Lemma \ref{teo1}.
We scale the solution around  {${\xi_k}_{\eps}$}  using the fastest concentration parameter ${\delta_k}_\eps$, i.e.  set
$$\tilde \phi_\eps(x):=(-1)^k{\delta_k}^{{N-2\over2}}_\eps \phi_\eps({\delta_k}_\eps x+{\xi_k}_\eps),\ x\in\tilde\Omega_\eps:={\Omega_\eps-{\xi_k}_\eps\over
{\delta_k}_\eps}=\left\{x\in\rr^N\ :\
 {\delta_k}_\eps x+{\xi_k}_\eps\in\Omega_\eps\right\},$$
This scaling allows to  {\em see} only the last bubble as it is shown in the next lemma.
\begin{lemma}\label{bub1}
It holds true that
$$ \tilde \phi_\eps(x) =   U(x)+\rho_\eps(x)+\tilde R_\eps(x),\ x\in\tilde\Omega_\eps$$
 with $U$ as in \eqref{U},
\begin{equation}\label{bub2}
\|  \rho_\eps\|_{L^{q}(\tilde\Omega_\eps)}\to 0 \ \hbox{for any}\ q>{2N\over N-2}\quad \hbox{and}\quad \| \tilde R_\eps\|_{L^{2N\over N-2}(\tilde\Omega_\eps)}\to 0\ \ \hbox{as}\ \eps\to0.
\end{equation}
\end{lemma}
\begin{proof}
We use \eqref{pro} and we get $ \tilde \phi_\eps(x) = U(x)+\rho_\eps(x)+\tilde R_\eps(x),$ where we set
$$\rho_\eps(x):={\delta_k}^{{N-2\over2}}_\eps\sum\limits_{i=1}^{k-1}(-1)^{i+k}  U_{{\delta_i}_\eps,{\xi_i}_\eps}({\delta_k}_\eps x+{\xi_k}_{\eps})\ \hbox{and}\ \tilde R_\eps(x):=(-1)^k{\delta_k}^{{N-2\over2}}_\eps
R_\eps({\delta_k}_\eps x+{\xi_k}_{\eps}).$$
It is immediate to check  that
$$\|  \tilde R_\eps\|_{L^{2N\over N-2}(\tilde\Omega_\eps)}=\|
R_\eps \|_{L^{2N\over N-2}( \Omega_\eps)}\to0\ \hbox{as}\ \eps\to0,$$
because of \eqref{resto}.
Moreover, we have
\begin{align*}
\int\limits_{\tilde\Omega_\eps}|\rho_\eps(x)|^qdx&\le c\sum\limits_{i=1}^{k-1} \int\limits_{\tilde\Omega_\eps}
\({\delta_k^{N-2\over2}\delta_i^{N-2\over2}\over\(\delta_i^2+|\delta_k x+\xi_k-\xi_i|^2\)^{N-2\over2}}\)^q
\ \hbox{(setting $x={\delta_i\over\delta_k} y+{\xi_i-\xi_k\over\delta_k}$)}
\\ &=c\sum\limits_{i=1}^{k-1}\({\delta_k\over\delta_i}\)^{{N-2\over2}q-N}\int\limits_{ \Omega_\eps-\xi_i\over\delta_i}{1\over\(1+  |y|^2\)^{{N-2\over2}q}}dy\
 \hbox{(we choose $\frac{q(N-2)}{2}>N$)}\\ &\le
 c\sum\limits_{i=1}^{k-1}\({\delta_k\over\delta_i}\)^{{N-2\over2}q-N}\int\limits_{ \rr^N}{1\over\(1+  |y|^2\)^{{N-2\over2}q}}dy\to 0\ \hbox{as}\ \eps\to0.\  \hbox{(because of \eqref{de}).}
\end{align*}

\end{proof}

 We consider the linearized operator at $\phi_\eps,$ that is
$$\mathcal L_\eps v:=-\Delta v-f'(\phi_\eps)v,\ v\in H^1_0(\Omega_\eps).$$
Let $\lambda_\eps$ the first eigenvalue of $\mathcal L_\eps$ and $\varphi_\eps$ the corresponding positive eigenfunction normalized in  $L^2(\Omega_\eps),$ i.e.
\begin{equation}\label{lin1}
-\Delta \varphi_\eps-f'(\phi_\eps)\varphi_\eps=\lambda_\eps\varphi_\eps\ \hbox{in}\ \Omega_\eps,\
 \varphi_\eps=0\ \hbox{on}\ \partial \Omega_\eps,\quad
\varphi_\eps>0\ \hbox{in}\ \Omega_\eps,\ \|\varphi_\eps\|_{L^2(\Omega_\eps)}=1.
\end{equation}
\begin{lemma}\label{lem0}
It holds true that $\lambda_\eps<0.$
\end{lemma}
\begin{proof}
It is enough to remark that by the definition of the first eigenvalue, i.e.
$$\lambda_\eps:=\min\limits_{v\in H^1_0(\Omega_\eps)\atop \|v\|_{L^2(\Omega_\eps)}=1} \int\limits_{\Omega_\eps} \(|\nabla v|^2-f'(\phi_\eps)v^2\)dx $$
 taking  the solution $v_\eps:=\phi_\eps/\|\phi_\eps\|_{L^2(\Omega_\eps) }$ as a test function we get
 $$\lambda_\eps\le{\int\limits_{\Omega_\eps} \(|\nabla \phi_\eps|^2-f'(\phi_\eps)\phi_\eps^2\)dx\over \int\limits_{\Omega_\eps}  \phi_\eps^2 dx}=(1-p_S){\int\limits_{\Omega_\eps}  |\phi_\eps|^{p_S+1} dx\over \int\limits_{\Omega_\eps}  \phi_\eps^2 dx}<0.$$
\end{proof}
Let us define
  $$
\tilde \varphi_\eps(x):={\delta_k}^{ {N \over2}}_\eps \varphi_\eps({\delta_k}_\eps x+{\xi_k}_{\eps})\ \hbox{if}\ x\in\tilde\Omega_\eps,\ \tilde\varphi_\eps(x):=0\ \hbox{if}\ x\not\in\tilde\Omega_\eps \ \hbox{and}\ \tilde\lambda_\eps:={\delta_k}_\eps^2\lambda_\eps.$$
Then, it is immediate to check that $\tilde\varphi_\eps$ solves
\begin{equation}\label{lin2}
-\Delta \tilde\varphi_\eps-f'(\tilde \phi_\eps)\tilde\varphi_\eps=\tilde\lambda_\eps\tilde\varphi_\eps\ \hbox{in}\ \tilde\Omega_\eps,\
\tilde\varphi_\eps=0\ \hbox{on}\ \partial\tilde\Omega_\eps,\quad
\tilde\varphi_\eps>0\ \hbox{in}\ \tilde\Omega_\eps,\ \|\tilde \varphi_\eps\|_{L^2(\rr^n)}=1.
\end{equation}
\begin{lemma}\label{lem1}
There hold true that
\begin{itemize}
\item[(i)]
There exists $c>0$ such that $\|\tilde\varphi_\eps\|_{H^1_0(\tilde\Omega_\eps)}\le c, $
\item[(ii)] $ \int\limits_{\tilde \Omega_\eps}|f'(\tilde \phi_\eps)-f'(U)|\tilde \varphi_\eps^2dx\to0$ as $\eps\to0.$
\end{itemize}
\end{lemma}
\begin{proof}
We have
\begin{align}\label{equ1}
\int\limits_{\tilde\Omega_\eps}|\nabla\tilde\varphi_\eps(x)|^2dx& =
\int\limits_{ \tilde \Omega_\eps}f'(\tilde \phi_\eps)\tilde \varphi_\eps^2dx+\tilde \lambda_\eps  \int\limits_{\tilde \Omega_\eps} \tilde \varphi_\eps^2dx \
 \hbox{(we use that $\tilde \lambda_\eps<0$ because of Lemma \ref{lem0})}\nonumber\\
&\le
\int\limits_{\tilde \Omega_\eps}f'(\tilde \phi_\eps)\tilde \varphi_\eps^2dx =\int\limits_{\tilde \Omega_\eps}f'(U)\tilde \varphi_\eps^2dx
+\int\limits_{\tilde \Omega_\eps}\(f'(\tilde \phi_\eps)-f'(U)\)\tilde \varphi_\eps^2dx
\end{align}
Now,
\begin{equation}\label{equ2}\int\limits_{\tilde \Omega_\eps}f'(U)\tilde \varphi_\eps^2dx\le \|f'(U)\|_{L^\infty(\rr^N)}\|\tilde\varphi_\eps\|_{L^2(\rr^N)}\le c
\end{equation}
for some positive constant $c.$
Moreover it is easy to check that
\begin{equation}\label{yyl}
\left||a+b|^{\alpha-1}(a+b)-|a|^{\alpha-1}a\right|\le \left\{\begin{aligned} &c \min \left\{|a|^{{\alpha}-1}|b|,|b|^{{\alpha}}\right\} \  \forall\ a,b\in\rr,\ \hbox{if}\ 0<{\alpha}\le 1,\\
&c \(|a|^{{\alpha}-1}|b|+|b|^{{\alpha}}\) \  \forall\ a,b\in\rr,\ \hbox{if}\ {\alpha}>1,
\end{aligned}\right.
\end{equation}
where $c$ is a positive constant only depending on ${\alpha}.$

Hence by \eqref{yyl} with ${\alpha}=p_S-1$ we get
$$ |f'(\tilde \phi_\eps)-f'(U)|
\le \left\{\begin{aligned} &c  |\tilde \phi_\eps-U|^{p_S-1} \   \hbox{ if }\ 1<p_S\le2,\\
&c \(|U|^{p_S-2}|\tilde \phi_\eps-U|+|\tilde \phi_\eps-U|^{p_S-1}\) \   \hbox{ if }\ p_S>2,
\end{aligned}\right.$$
Then if $p_S\le2$, i.e. $N\ge6$, we get
\begin{align}\label{equ3}
 \int\limits_{\tilde \Omega_\eps}|f'(\tilde \phi_\eps)-f'(U)|\tilde \varphi_\eps^2dx&\le c \int\limits_{\tilde \Omega_\eps}|\tilde \phi_\eps-U|^{p_S-1}\tilde \varphi_\eps^2dx
 \ \hbox{(we use Lemma \ref{bub1})}
 \nonumber\\ &
\le c \int\limits_{\tilde \Omega_\eps}\(|\rho_\eps |^{p_S-1}+|\tilde R_\eps|^{p_S-1}\) \tilde \varphi_\eps^2dx\ \hbox{(by H\"older's inequality with $t>N/2$)}\nonumber \\ &
\le c  \|\rho_\eps \|^{p_S-1}_{L^{(p_S-1)t}(\tilde\Omega_\eps)}\|\tilde \varphi_\eps \|^{2}_{L^{2t\over t-1}(\tilde\Omega_\eps)}
+c\|\tilde R_\eps \|^{p_S-1}_{L^{2N\over N-2}(\tilde\Omega_\eps)}\|\tilde \varphi_\eps \|^{2}_{L^{2N\over N-2}(\tilde\Omega_\eps)}
\end{align}
and if $p_S>2,$ i.e. $N=3,4,5$, we get
\begin{align}\label{equ4}
 \int\limits_{\tilde \Omega_\eps}|f'(\tilde \phi_\eps)-f'(U)|\tilde \varphi_\eps^2dx&\le c \int\limits_{\tilde \Omega_\eps}\(|U|^{p_S-2}|\tilde \phi_\eps-U|+|\tilde \phi_\eps-U|^{p_S-1}\) \tilde \varphi_\eps^2dx
 \ \hbox{(we use Lemma \ref{bub1})}\nonumber\\ &
\le c \int\limits_{\tilde \Omega_\eps}\left[|U|^{p_S-2}\(|\rho_\eps|+|\tilde R_\eps |\)+|\rho_\eps |^{p_S-1}+|\tilde R_\eps|^{p_S-1}\right] \tilde \varphi_\eps^2dx\nonumber\\
&\qquad \hbox{(we apply H\"older's inequality with   $t>N/2$)} \nonumber \\ &
\le c  \|U \|^{p_S-2}_{L^{\infty}(\rr^N)}\|\rho_\eps \|_{L^{t}(\tilde\Omega_\eps)}\|\tilde \varphi_\eps \|^{2}_{L^{2t\over t-1}(\tilde\Omega_\eps)}
\nonumber\\ &\quad
+c\|U \|^{p_S-2}_{L^{2N\over 6-N}(\tilde\Omega_\eps)}\|\tilde R_\eps \| _{L^{2N\over N-2}(\tilde\Omega_\eps)}
\|\tilde \varphi_\eps \|^{2}_{L^{2N\over N-2}(\tilde\Omega_\eps)}\nonumber\\ &\quad +
 c  \|\rho_\eps \|^{p_S-1}_{L^{(p_S-1)t}(\tilde\Omega_\eps)}\|\tilde \varphi_\eps \|^{2}_{L^{2t\over t-1}(\tilde\Omega_\eps)}
+c\|\tilde R_\eps \|^{p_S-1}_{L^{2N\over N-2}(\tilde\Omega_\eps)}\|\tilde \varphi_\eps \|^{2}_{L^{2N\over N-2}(\tilde\Omega_\eps)}
\end{align}
Now, we remark that
$${2t\over t-1}=2\theta+{2N\over N-2}(1-\theta)\quad \hbox{with}\ \theta:={2t-N\over2(t-1)}\in(0,1)$$
and so by interpolation
\begin{equation}\label{equ5}\|\tilde \varphi_\eps \|_{L^{2t\over t-1}(\tilde\Omega_\eps)}\le \|\tilde \varphi_\eps \|^{{ \theta(t-1)\over t}}_{L^{2 }(\tilde\Omega_\eps)}\
 \|\tilde \varphi_\eps \|^{N\over 2t}_{L^{2N\over N-2}(\tilde\Omega_\eps)}\le \|\tilde \varphi_\eps \|^{{ N\over 2t} }_{L^{2N\over N-2}(\tilde\Omega_\eps)}
 \end{equation}
 because $ \|\tilde \varphi_\eps \|_ {L^{2 }(\tilde\Omega_\eps)}\le1.$
 Finally, we collect \eqref{equ1}--\eqref{equ5}, we use Sobolev's inequality
 $\|\tilde \varphi_\eps \| _{L^{2N\over N-2 }(\tilde\Omega_\eps)}       \le c\|\tilde \varphi_\eps \| _{H^1_0(\tilde\Omega_\eps)}$, we also
 estimate \eqref{bub2} and we get
 $$(1-\alpha(\eps))\|\tilde \varphi_\eps \|^2_{H^1_0(\tilde\Omega_\eps)}\le c+\beta(\eps) \|\tilde \varphi_\eps \|^{{ N\over  t} }_{H^1_0(\tilde\Omega_\eps)}$$
 where $\alpha(\eps),\beta(\eps)\to0$ as $\eps\to0.$ Therefore (i) follows, because  $t>N/2$.\\
 By (i), by estimates \eqref{equ3}, \eqref{equ4} and \eqref{equ5} and by \eqref{bub2} we immediately get (ii).
\end{proof}

\begin{lemma}\label{lem2}
It holds true that
\begin{itemize}
\item[(i)] $\lim\limits_{\eps\to0}\tilde\lambda_\eps=\lambda^* $
\item[(ii)] $ \varphi_\eps $ strongly converges to $\varphi^*$ in $L^2(\rr^N)$ as $\eps\to0.$
\end{itemize}
\end{lemma}
\begin{proof}
Let us prove (i). By the definiton of $\lambda^*$ and by \eqref{lin2}, we get
\begin{align}\label{equ7}
\lambda^*&\le \int\limits_{\rr^N}\(|\nabla\tilde\varphi_\eps|^2-f'(U)\tilde\varphi_\eps^2\)dx=\int\limits_{\tilde\Omega_\eps}
\(|\nabla\tilde\varphi_\eps|^2-f'(U)\tilde\varphi_\eps^2\)dx\nonumber\\ &=
\int\limits_{\tilde\Omega_\eps}\(|\nabla\tilde\varphi_\eps|^2-f'(\tilde \phi_\eps)\tilde\varphi_\eps^2\)dx+\int\limits_{\tilde\Omega_\eps}\( f'(\tilde \phi_\eps)-f'(U)\)
\tilde\varphi_\eps^2 dx\nonumber\\
&=\tilde\lambda_\eps+\alpha(\eps),\ \hbox{where   $\alpha(\eps)\to0$ as $\eps\to0$, }
\end{align}
because of (ii) of Lemma \ref{lem1}. \\
On the other hand, let us consider a regular cut-off function $\chi_\eps(x)=\chi_\eps(|x-x_0|)$ such that $0\le \chi_\eps\le 1 $  and
 $$\chi_\eps(r)=1\ \hbox{if}\ 4\eps\le r\le {{\hbox{diam}  (\Omega)}\over4}\quad \hbox{and}\ \chi_\eps(|x|)=0\ \hbox{if}\ r\le 2\eps\ \hbox{or}\  r\ge {{\hbox{diam}  (\Omega})\over2}$$
  Let us consider the functions
$$w_\eps(x):={\chi_\eps(x)\varphi^*\({(x-{\xi_k}_\eps)/{\delta_k}_\eps} \)\over\|\chi_\eps(x)\varphi^*((x-{\xi_k}_\eps)/{\delta_k}_\eps )\|_{L^2(\Omega_\eps)}},\
x\in\Omega_\eps\ \hbox{and}\
\tilde w_\eps(y):={\chi_\eps({\delta_k}_\eps y+{\xi_k}_\eps)\varphi^*(y )\over\|\chi_\eps({\delta_k}_\eps y+{\xi_k}_\eps)\varphi^*(y )\|_{L^2(\tilde \Omega_\eps)}},\ y\in\tilde\Omega_\eps.$$
It is easy to check that
\begin{equation}
\label{equ6}
\tilde w_\eps\to\varphi^*\ \hbox{in}\ H^1(\rr^N)\ \hbox{as}\ \eps\to0.\end{equation}
By the definition of $\lambda_\eps$ and scaling $x= {\delta_k}_\eps y+{\xi_k}_{\eps}$ we get
$$\lambda_\eps\le \int\limits_{\Omega_\eps}\(|\nabla w_\eps|^2-f'(\phi_\eps)w^2_\eps\)dx={1\over{\delta_k}_\eps^2}
\int\limits_{\tilde \Omega_\eps}\(|\nabla \tilde w_\eps|^2-f'(\tilde \phi_\eps)\tilde w^2_\eps\)dy,$$
which implies
\begin{align}\label{equ8}\tilde \lambda_\eps&\le
\int\limits_{\tilde \Omega_\eps}\(|\nabla \tilde w_\eps|^2-f'(\tilde \phi_\eps)\tilde w^2_\eps\)dy=
\int\limits_{\rr^N}\(|\nabla \tilde w_\eps|^2-f'(\tilde \phi_\eps)\tilde w^2_\eps\)dy\nonumber\\
&=\int\limits_{\rr^N}\(|\nabla \tilde w_\eps|^2-f'(U)\tilde w^2_\eps\)dy+\int\limits_{\rr^n}\(f'(U)-f'(\tilde \phi_\eps)\)\tilde w^2_\eps dy\nonumber\\
&= \lambda^*+\beta(\eps),\ \hbox{where   $\beta(\eps)\to0$ as $\eps\to0$, }
\end{align}
because by \eqref{equ6} we deduce that
$$\int\limits_{\rr^N}\(|\nabla \tilde w_\eps|^2-f'(U)\tilde w^2_\eps\)dy\to \int\limits_{\rr^N}\(|\nabla \varphi^*|^2-f'(U){\varphi^*}^2\)dy=\lambda^* \ \hbox{as}\ \eps\to0$$
and arguing exactly as in the proof of (ii) of Lemma \ref{lem1} we get
$$\int\limits_{\rr^n}\(f'(U)-f'(\tilde \phi_\eps)\)\tilde w^2_\eps dy\to 0\ \hbox{as}\ \eps\to0.$$
Finally, by \eqref{equ7} and \eqref{equ8} the claim follows.\\
\\
Next we prove (ii).
By the definition of $\tilde\lambda_\eps$ and (i) we have
$$\int\limits_{\widetilde \Omega_\eps}\(|\nabla\tilde\varphi_\eps\|^2-f'(\tilde\phi_\eps)\tilde\varphi_\eps^2\)dy=\tilde\lambda_\eps\to\lambda^*\ \hbox{as}\ \eps\to0,$$
which implies that $\tilde\varphi_\eps$ is a minimizing sequence for \eqref{las} and so the claim follows by Lemma \ref{3.3}.
\end{proof}

\

\begin{proof}[Proof of Proposition \ref{towerVerificaCondizioneIntegraleCasoGEN}]

We now prove that
$$\liminf\limits_{\eps\to0}\int\limits_{\Omega_\eps}\phi_\eps\varphi_\eps dx>0.$$

We multiply equation \eqref{problem} by $\varphi_\eps$ and equation \eqref{lin1} by $\phi_\eps$,
we subtract the two equations and we get
\begin{equation}\label{eq1}\int\limits_{\Omega_\eps}\phi_\eps\varphi_\eps dx=-{p_S-1\over\lambda _\eps}\int\limits_{\Omega_\eps}f(\phi_\eps)\varphi_\eps dx.\end{equation}
Therefore, we are lead to study the sign of the right hand side of \eqref{eq1}. We are going to prove that
 \begin{equation}\label{eq2}\lim\limits_{\eps\to0}{\delta_k}_\eps\int\limits_{\Omega_\eps}f(\phi_\eps)\varphi_\eps dx=
 \int\limits_{\rr^N}f(U)\varphi^* dx.
 \end{equation}
Since the right hand side of \eqref{eq2} is positive, this will imply that the right hand side of \eqref{eq1} is positive and finally the claim will follow.
\\
Let us prove \eqref{eq2}.
We have
\begin{align}\label{eq3}
  {\delta_k}_\eps\int\limits_{\Omega_\eps}f(\phi_\eps)\varphi_\eps dx- \int\limits_{\rr^N}f(U)\varphi^* dx&=
 \int\limits_{\tilde \Omega_\eps}f(\tilde \phi_\eps)\tilde\varphi_\eps dx- \int\limits_{\rr^N}f(U)\varphi^* dx\nonumber\\
 &= \int\limits_{\tilde \Omega_\eps}\left[f(\tilde \phi_\eps)-f(U)\right]\tilde \varphi_\eps dx +  \int\limits_{\rr^N}f(U)\left[\tilde\varphi_\eps-\varphi^*\right] dx .
\end{align}
By H\"older's inequality we get
$$\left|\int\limits_{\rr^N}f(U)\left[\tilde\varphi_\eps-\varphi^*\right] dx\right|\le \left\|f(U)\right\|_{L^2(\rr^N)}  \left\|\tilde\varphi_\eps-\varphi^*\right\|_{L^2(\rr^N)}\to 0\ \hbox{as}\ \eps\to0$$
because of (ii) of Lemma \ref{lem2}.
Moreover, by \eqref{yyl} with ${\alpha}=p_S$   and using H\"older's inequality we get (here we choose $q>{2N\over N-2}$ with  $q(N-8)<2N$)
\begin{align*}
   \left| \int\limits_{\tilde \Omega_\eps}\left[f(\tilde \phi_\eps)-f(U)\right]\tilde \varphi_\eps dx\right|&\le  c
  \int\limits_{\tilde \Omega_\eps}  \( f'(U) |\tilde \phi_\eps-U|+|\tilde \phi_\eps-U|^{p_S}\)|\tilde \varphi_\eps| dx \\&\leq  c
  \int\limits_{\tilde \Omega_\eps}  \left[ U^{4\over N-2 }\( |\rho_\eps|+|\tilde R_\eps |\)+\(|\rho_\eps|^{N+2\over N-2}+|\tilde R_\eps |^{N+2\over N-2}\)\right]|\tilde \varphi_\eps| dx \\
  &\le c \left\| U \right\|^{4\over N-2}_{L^{8q\over (q-2)(N-2)}(\rr^N)}\left\| \rho_\eps \right\|_{L^{q}(\tilde \Omega_\eps)}
  \left\| \tilde \varphi_\eps \right\|_{L^{2}(\tilde \Omega_\eps)}\\ &\quad+c \left\| U \right\|^{4\over N-2}_{L^{2N\over N-2}(\rr^N)}\left\| \tilde R_\eps \right\|_{L^{2N\over N-2}(\tilde \Omega_\eps)}
  \left\| \tilde \varphi_\eps \right\|_{L^{2N\over N-2}(\tilde \Omega_\eps)}
  \\ &\quad+c\left\| \rho_\eps \right\|^{\frac{N+2}{N-2}}_{L^{2(N+2)\over N-2}(\tilde \Omega_\eps)}
  \left\| \tilde \varphi_\eps \right\|_{L^ {2}(\tilde \Omega_\eps)} +c\left\| \tilde R_\eps \right\|^{N+2\over N-2}_{L^{2N\over N-2}(\tilde \Omega_\eps)}
  \left\| \tilde \varphi_\eps \right\|_{L^ {2N\over N-2}(\tilde \Omega_\eps)} \\ &\quad\hbox{(we apply Lemma \ref{lem1} and estimate \eqref{bub2}}\\ &\le \alpha(\eps)
\end{align*}
 where $\alpha(\eps) \to0$ as $\eps\to0 $. Therefore the claim follows.
\end{proof}

\

\section{Proof of Theorem \ref{teorema:casoRadiale}}
The proof of Theorem \ref{teorema:casoRadiale} consists in
finding a radial sign changing stationary solution to \eqref{problem} which satisfies \eqref{condizioneProdottoScalareDiversoDaZeroGen}.
Then Proposition \ref{proposition:generaleCazenaveDickWeiss} applies and the conclusion follows from the last statement of Proposition \ref{proposition:comparison}.

As the radial sign changing stationary solution we take now a radial $k$-tower solution, which exists in the annulus with a sufficiently small hole, as next result asserts:

\begin{lemma}[Existence of radial $k$-tower stationary solutions]
\label{teo1}\label{existenceOfRadialTowers}
Let $\Omega_{\epsilon}:=\{x\in\mathbb R^N: 0<\epsilon<|x|<1\}$,  $N\in\mathbb N$, $N\geq 3$ and $\epsilon >0$ small.
For any integer $k\ge {2}$ there exists $\eps_k>0$ such that for any $\eps\in(0,\eps_k)$ problem \eqref{problem} has  a radial $k$-tower  {sign-changing} stationary solution $\widehat{\phi_\eps}$ whose profile is
\begin{equation}\label{Radpro}
\widehat{\phi_\eps}(x)=\sum\limits_{i=1}^k(-1)^{i}  U_{{\delta_i}_\eps,0}(x)+\widehat{R_\eps}(x),\ x\in\Omega_\eps
\end{equation}
where the concentration parameters ${\delta_i}_\eps$'s satisfy
\begin{equation}\label{radde}
{\delta_i}_\eps:={\widehat{d_i}}_\eps\eps^{2i-1\over 2k},\ {\widehat{d_i}}_\eps\in\mathbb R\quad \hbox{and}\quad {\widehat{d_i}}_\eps\to \widehat{d_i}>0\ \hbox{as}\ \eps\to0\quad
\hbox{for}\  i=1,\dots,k,
\end{equation}
and the remainder term $\widehat{R_\eps}$ is radial and satisfies
\begin{equation}\label{resto}
\|\widehat{R_\eps}\|_{L^{2N\over N-2}(\Omega_\eps)}\to 0\ \ \hbox{as}\ \eps\to0.
\end{equation}
\end{lemma}

\begin{proof}
Combining the ideas in \cite{MussoPistoia} with the general arguments in \cite{gmp} one can prove that
for any integer $k\ge1$
there exists $\eps_k>0$ such that for any $\eps\in(0,\eps_k)$ problem \eqref{problem} has  a stationary solution $\widehat{\phi_\eps}$ whose profile is
\begin{equation}\label{pq1rad}
\widehat{\phi_\eps}(x)=
\sum\limits_{i=1}^k(-1)^i  P_{\Omega_\eps}U_{{\delta_i}_\eps,0}(x)  +\widehat{\psi_\eps}(x),\ x\in\Omega_\eps
\end{equation}
where  the concentration parameters $\delta_1={\delta_1}_\eps,\dots,\delta_k={\delta_k}_\eps$  satisfy
\eqref{radde}
and   the remainder term $\widehat{\psi_\eps}$ is radial and satisfies
\begin{equation}\label{pq2rad}
\|\widehat{\psi_\eps}\|_{H^1_0(\Omega_\eps)}\to0\ \hbox{as}\ \eps\to0.
\end{equation}
The conclusion follows similarly as in the proof of Lemma \ref{existenceOfTowers}.
\end{proof}

\

Performing an asymptotic spectral analysis similarly as in the proof of Proposition \ref{towerVerificaCondizioneIntegraleCasoGEN}  we can now prove the following result and hence conclude

\begin{proposition}\label{towerRadialeVerificaCondizioneIntegraleCasoGEN}
Let $\Omega_{\epsilon}:=\{x\in\mathbb R^N: 0<\epsilon<|x|<1\}$,  $N\in\mathbb N$, $N\geq 3$ and $\epsilon >0$.
Let $\widehat{\phi_{\epsilon}}$ be as in Proposition \ref{existenceOfRadialTowers}.
There exists $\epsilon_0>0$ such that for any $\epsilon\in (0,\epsilon_0)$
\[\int_{\Omega_{\epsilon}}\widehat{\phi_{\epsilon}}\varphi_{1,\epsilon}dx >0,\]
where $\varphi_{1,\epsilon}$ is the positive eigenfunction of the self-adjoint operator $L_{\epsilon}=-\Delta -p_S|\widehat{\phi_{\epsilon}}|^{p_S-1}$ on $L^2(\Omega_{\epsilon})$ with domain $ H^2(\Omega_{\epsilon})\cap H^1_0(\Omega_{\epsilon})$.
\end{proposition}

\

\section*{Acknowledgments}
Portions of this research were done while the first author was visiting the second one. The first author would like to thank the {\it Facultad de Matem{\'a}ticas} of the {\it Pontificia Universidad Cat{\'o}lica de Chile} for the warm hospitality during her staying. The authors would like to express their gratitude to Thierry Cazenave for many helpful comments and remarks.
\

\

\end{document}